\documentclass{amsart}
\usepackage[utf8]{inputenc}
\usepackage{amssymb,bm,amsfonts, stmaryrd, amsmath, amsthm, enumerate, mathtools, hyperref}
\usepackage{wrapfig,tikz, tikz-cd}
\usetikzlibrary{arrows}

\newtheorem{theorem}{Theorem}[section]
\newtheorem{lemma}[theorem]{Lemma}
\newtheorem{proposition}[theorem]{Proposition}

\newtheorem{theorema}{Theorem}
\newtheorem*{theorem*}{Theorem}

\newtheorem{theoremaprime}{Theorem}

\theoremstyle{definition}
\newtheorem{construction}[theorem]{Construction}

\newtheorem{chunk}[theorem]{}

\usepackage{comment}

\newcommand{\Hom}{{\operatorname{Hom}}}
\newcommand{\Der}{{\operatorname{Der}}}
\newcommand{\Ext}{{\operatorname{Ext}}}

\renewcommand{\H}{\operatorname{H}}

\newcommand{\ad}{{\operatorname{ad}}}
\newcommand{\tors}{{\operatorname{t}}}
\newcommand{\lotimes}{\otimes^{\sf L}}

\newcommand{\susp}{{\scriptstyle\mathsf{\Sigma}}}

\newcommand{\pd}{\operatorname{projdim}}

\newcommand{\ind}{\operatorname{ind}}
\newcommand{\m}{\mathfrak{m}}
\newcommand{\dd}[1]{\frac{{\rm d}}{{\rm d}{#1}}}

\title[Vasconcelos' conjecture]{Vasconcelos' conjecture on the conormal module}

\author{Benjamin Briggs}
\address{University of Utah, Department of Mathematics, 
Salt Lake City, UT 84112}
\email{briggs@math.utah.edu}

\begin{document}


\begin{abstract}

For any ideal $I$ of finite projective dimension in a commutative noetherian local ring  $R$, we prove that if the conormal module $I/I^2$ has finite projective dimension over $R/I$, then $I$ must be generated by a regular sequence. 
This resolves a conjecture of Vasconcelos. 
We prove a similar result for the first Koszul homology module of $I$. When $R$ is a localisation of a polynomial ring over a field $K$ of characteristic zero, Vasconcelos conjectured that $R/I$ is a reduced complete intersection if the module $\Omega_{(R/I)/K}$ of K\"ahler differentials has finite projective dimension; we prove this contingent on the Eisenbud-Mazur conjecture.

The arguments exploit the structure of the homotopy Lie algebra associated to $I$ in an essential way. By work of Avramov and Halperin, if every degree $2$ element of the homotopy Lie algebra is radical, then $I$ is generated by a regular sequence. Iyengar has shown that free summands of $I/I^2$ give rise to central elements of the homotopy Lie algebra, and we establish an analogous criterion for constructing radical elements, from which we deduce our main result.


\end{abstract}

\dedicatory{Dedicated to the memory of Wolmer Vasconcelos}

\maketitle

\pagenumbering{arabic}

Among the first invariants one might attach to an ideal $I$ in a commutative noetherian ring $R$ is the {conormal module} $I/I^2$ over the quotient $S=R/I$. Ring theoretic properties of $S$ are often reflected in module theoretic properties of $I/I^2$. This is well illustrated by a result of Ferrand and Vasconcelos: as long as $I$ has finite projective dimension, the conormal module is projective over $S$ if and only if $I$ is locally generated by a regular sequence 
\cite{MR219546,MR213345}.

Vasconcelos later made a substantially stronger conjecture: \emph{if both $\pd_R I$ and $\pd_S I/I^2$ are finite, then $I$ is locally generated by a regular sequence} \cite{MR508082}. As evidence he offered the case $\pd_S I/I^2\leq 1$, and several special cases of low height \cite{MR0262227, MR814190}.



Gradual progress was made in the years that followed; 
Vasconcelos surveyed what was known at the time in \cite{MR814190}, as did Herzog in \cite{brazil}. 
The problem has also inspired a number of interesting research directions in the decades since (\cite{MR1052871,MR2989999,MR595423,MR1743666,MR882705} to name a few). 
A major forward step was taken by Avramov and Herzog   in \cite{MR1269426}; by analysing the behaviour of the Euler derivation, they established the conjecture for graded algebras over fields of characteristic zero. 


We resolve Vasconcelos' conjecture in its full generality.


\begin{theorema}\label{VC}
Let $R\to S=R/I$ be a surjective homomorphism of commutative noetherian rings, with $S$ of finite projective dimension over $R$. If the conormal module $I/I^2$ has finite projective dimension over $S$, then $I$ is locally generated by a regular sequence.
\end{theorema}

In the case that $R$ is a regular local ring, Theorem \ref{VC} yields a new characterisation of local complete intersection rings. 
The theorem can be restated geometrically:

\begin{theoremaprime}
Let $f\colon X\to Y$ be a closed immersion of locally noetherian schemes, locally of finite flat dimension. If the conormal sheaf of $f$ is perfect, then $f$ is a regular immersion.
\end{theoremaprime}

We also prove a result for the first Koszul homology, analogous to Theorem \ref{VC}.

\begin{theorema}\label{Kosthm}
Let $I$ be an ideal of finite projective dimension in a local ring $R$. If the first Koszul homology $\H_1(I;R)$ has finite projective dimension over $R/I$, then $I$ is generated by a regular sequence.
\end{theorema}




When $S$ is a generically separable algebra, essentially of finite type over a field $K$, Ferrand proved that  $S$ is a reduced complete intersection if and only if the module $\Omega_{S/K}$  of K\"ahler differentials has projective dimension at most one \cite{MR219546}.  Accordingly, Vasconcelos conjectured that $S$ is a reduced complete intersection as soon as $\Omega_{S/K}$ has finite projective dimension \cite{MR508082}. This is known to hold in all of the cases mentioned above concerning the conormal module  \cite{MR1269426,brazil,MR814190}. We prove that this conjecture is a consequence of Theorem \ref{VC} and the Eisenbud-Mazur conjecture, and thereby establish a number of new cases---see Section \ref{Kahlersection}.


The key tool in our proofs is the  {homotopy Lie algebra} $\pi^*(\varphi)$, and the structure of its radical. This is a graded Lie algebra naturally associated with any local homomorphism $\varphi\colon R\to S$, so-called by analogy with its namesake in rational homotopy theory. Its use in commutative algebra has been championed by Avramov, Halperin and others \cite{MR749041,MR846435}.

Since this paper first appeared, Iyengar and the author  have built on the ideas used here to obtain new rigidity results on the cotangent complex \cite{2020arXiv201013314B}, subsuming  Theorem \ref{VC} as well as the work of Avramov  on Quillen’s conjecture \cite{MR1726700}. 

\subsection*{Outline}
After some homological background in Section \ref{prelimSection}, most of the work takes place in Section \ref{radsection}, where we analyse the radical of $\pi^*(\varphi)$. Then in Section \ref{testsection} we present various test modules which characterise complete intersections by the finiteness of their projective dimension. Theorems \ref{VC} and \ref{Kosthm} are proven here, along with our results on 
the module of K\"ahler differentials.
 

\subsection*{Acknowledgements} I am indebted to Srikanth Iyengar for his close attention to this paper from its first draft to its last, and for his many important suggestions along the way. The ideas here arose after talking to Elo\'isa Grifo and Josh Pollitz about the homotopy Lie algebra, and I am happy to thank them as well. I am also grateful to the referee for a number of important suggestions which improved the readability of this paper.

\section{Preliminaries}\label{prelimSection}

We fix once and for all a surjective local homomorphism  $\varphi\colon R\to S=R/I$ of noetherian local rings with common residue field $k$.

The maximal ideal of $R$ is $\m_R$, and the same schema will go for other local rings and local graded rings. Our notation for the degree-shift of a graded object $M$ is $(\susp M)_i=M_{i-1}$. This comes with a tautological degree $1$ map $\susp\colon M\to \susp M$.  If $M$ happens to be a graded vector space over $k$, then its graded dual is denoted $M^\vee=\Hom_k(M,k)$.

The reference \cite{MR2641236} contains everything we need about differential graded algebras and modules (henceforth dg algebras and modules), including a great deal more information on minimal models.

\subsection{Minimal models}\label{minmod} A \emph{minimal model} for $\varphi$ is a factorisation $R\to A\to S$, where $A$ is a dg $R$-algebra with the following properties
\begin{enumerate}
    \item 
    $A=R[X]$ is the free strictly graded commutative $R$-algebra on a graded set $X=X_1,X_2,...\,$, each $X_i$ being a set of degree $i$ variables;
    \item 
    the differential of $A$ satisfies $\partial(\m_A)\subseteq \m_RA+\m_A^2$;
    \item 
    $A\to S$ is a quasi-isomorphism.
\end{enumerate}
These properties determine $A$ up to an isomorphism of dg $R$-algebras over $S$ \cite{MR2641236}.  Minimal models are the local commutative algebra analogue of Sullivan models \cite{MR935009}.


More generally we consider the finite stages  $A_{(n)}=R[X_{< n}]$. The directed system $R=A_{(1)}\to A_{(2)}\to \cdots$ has $A$ as its colimit; this is analogous to the Postnikov tower in rational homotopy theory. The fibre of the inclusion $A_{(n)}\to A$ is  denoted  
\[
A^{(n)}\coloneqq A\otimes_{A_{(n)}}k=A/(\m_R,X_{<n})=k[X_{\geq n}].
\]
The first of these, $A^{(1)}$, is significant  because it is the derived fibre $S\lotimes_R k$ of $\varphi$. 

\subsection{Derivations}\label{DerChunk}
An {$R$-linear derivation} is a $R$-linear homomorphism $\theta\colon A\to M$, homogeneous of some degree $i$, satisfying the graded Leibniz rule \[
\theta(xy)=\theta(x)y+(-1)^{ij}x\theta(y)
\]
for $x$ in $A_j$ and $y$ in $A$. 

We do not assume derivations are chain maps---those that do commute with the differentials are called \emph{chain derivations}. Taken all together, the $R$-linear derivations form a complex $\Der_R(A,M)$ with differential  $\partial(\theta)=\partial_M \theta -(-1)^{|\theta|}\theta \partial_A $.

Taking  $A$ to be a minimal model as in \ref{minmod},  any $R$-linear derivation $\theta\colon A\to A$ of non-positive degree  preserves the ideal $(\m_R,X_{<n})$. This implies that $\theta$ induces a derivation on each of the fibres  $\theta^{(n)}\colon A^{(n)}\to A^{(n)}$.

\subsection{K\"{a}hler differentials}\label{KahlerChunk}
The dg $A$-module of K\"{a}hler differentials $\Omega_{A/R}$ 
is defined by the isomorphism $\Der_R(A,M)\cong \Hom_A(\Omega_{A/R},M)$, natural in the dg $A$-module $M$. This isomorphism is realised by composition with the universal $R$-linear derivation $d:A\to \Omega_{A/R}$.

We will need the following facts:
\begin{enumerate}
    \item\label{K1} 
    $\Omega_{A/R}$ is a minimal, semi-free dg $A$-module. The underlying $A$-module is free on a basis $dX_i$ in degree preserving bijection with $X_i$; see   \cite[(1.13)]{MR1269426}. 
    \item \label{K2} 
    The projection  $\Omega_{A/R}\to S\otimes_A\Omega_{A/R}$ is a quasi-isomorphism, and $S\otimes_A\Omega_{A/R}$ is a minimal complex of free $S$-modules, isomorphic as a graded $S$-module to $\bigoplus S dX_i$. This follows from (\ref{K1}).
    \item \label{K3} The connection to the conormal module is  $\H_1(\Omega_{A/R})\cong \H_1(S\otimes_A\Omega_{A/R})\cong I/I^2$; see \cite[(2.5)]{MR1269426}.
\end{enumerate}
When $R$ contains the rational numbers $S\otimes_A\Omega_{A/R}$ is none other than the cotangent complex defined by Andr\'e and Quillen \cite{MR896094}. In general, $S\otimes_A\Omega_{A/R}$ is a different object.


\subsection{The homotopy Lie algebra}\label{hoLieDef}
 The homotopy Lie algebra is encoded quite directly in the minimal model $A$ for $\varphi$, just as the rational homotopy Lie algebra of a space can be seen in its Sullivan model.  It can be constructed using the following recipe:

\begin{enumerate}
    \item 
    Let $\ind A^{(1)}\coloneqq \m_{A^{(1)}}/\m_{A^{(1)}}^{2}$ be the graded vector space of indecomposables of $A^{(1)}$, with basis $X_i$ in degree $i$. Then define
    \[
    \pi^*(\varphi)\coloneqq(\susp \ind A^{(1)})^\vee,
    \]
    so that  $\pi^i(\varphi)$ has a basis dual to $X_{i-1}$.
    \item\label{bracketdef} Let  $\ind^2 A^{(1)}\coloneqq \m_{A^{(1)}}^2/ \m_{A^{(1)}}^{3}$. This has a basis of consisting of the   monomials $xy$ with $x$ and $y$ in $X$. 
    The pairing $\langle u,x\rangle =u(\susp x)$ between $\pi^*(\varphi)$ and $\ind A^{(1)}$ extends to a pairing between $\pi^*(\varphi)\otimes \pi^*(\varphi)$ and $\ind^2 \!A^{(1)}$:
    \[
    \langle u\otimes v, xy\rangle  \coloneqq \langle v,x\rangle\langle u,y\rangle+ (-1)^{(i+1)(j+1)}  \langle u,x\rangle\langle v,y\rangle
    \]
    for $u$ in $ \pi^i(\varphi)$ and  $v$ in  $\pi^j(\varphi)$. 
    Since $A$ is a minimal model, its differential induces a map  $\overline{\partial}\colon \ind A^{(1)}\to \ind^2\!A^{(1)}$, and the bracket $\pi^*(\varphi)\otimes \pi^*(\varphi)\to \pi^*(\varphi)$ is  dual to  $\overline{\partial}$ by definition: 
    \[
     \langle [u,v] , x\rangle = (-1)^j\langle u\otimes v, \overline{\partial}(x) \rangle.
    \]
\end{enumerate}
A graded Lie algebra should also have a quadratic square operation $\pi^i(\varphi)\to \pi^{2i}(\varphi)$ defined for odd $i$; since we will not need this structure we just refer to \cite{MR2641236,MR896094}. That $\pi^*(\varphi)$ with all this structure satisfies the axioms of a graded Lie algebra was proven by Avramov \cite{MR749041,briggs2018}.

If $\varphi$ is a complete intersection then $\pi^*(\varphi)$ is an abelian Lie algebra concentrated in degree $2$; otherwise its structure is highly non-trivial.

The existence of the homotopy Lie algebra is an instance of the Koszul duality between commutative algebras and Lie algebras. It is a remarkable theorem of Avramov that the universal enveloping algebra $U\pi^*(\varphi)$ is canonically isomorphic to  $\Ext^*_{A^{(1)}}(k,k)$ \cite{MR749041,briggs2018}. This isomorphism is often used to define $\pi^*(\varphi)$, but the minimal model approach above is more suited to our needs.


The Lie subalgebra $\pi^{>n}(\varphi)$ can be obtained in the same way using $(\susp \ind A^{(n)})^\vee$. For this reason one can think of $A^{(n)}$ as the {$n$-connected cover} of $A$, by analogy with rational homotopy theory.

It follows from the definition that there is a canonical isomorphism of graded vector spaces
\[
\Der_R(A,k)\cong \susp \pi^*(A),
\]
although in these terms it is less clear how to define the Lie bracket.

\section{The radical of the homotopy Lie algebra}\label{radsection}

%

Any graded Lie algebra $L$ has an inductively defined derived series $L^{[0]}=L$, $L^{[n+1]}=[L^{[n]},L^{[n]}]$, and $L$ is called solvable if there is an $n$ such that $L^{[n]}=0$. An ideal of $L$ is called solvable if it is solvable as a Lie algebra in its own right, and the radical of $L$ is the sum of all its solvable ideals. An element of $L$ is called radical if it generates a solvable ideal, or equivalently, if it lies in the radical of $L$. See \cite{MR935009} for more information.

Henceforth our standing assumption is that $\varphi\colon R\to S$ is a surjective local homomorphism of finite projective dimension, with kernel $I$. In this context, the authors of \cite{MR935009} prove that the radical of $\pi^*(\varphi)$ is  finite dimensional. We will use the following characterisation:

    \begin{chunk}\label{RadDef}
    \emph{An element $z$ of $\pi^*(\varphi)$ is in the radical  if and only if for some $n$ the restriction ${\rm ad}(z)=[z,-]\colon\pi^{>n}(\varphi)\to \pi^{>n}(\varphi)$ is zero}.
    \end{chunk}
    
    One direction is elementary: if we assume $[z,\pi^{>n}(\varphi)]=0$ then the derived series of the ideal $(z)$ will terminate before $n$ steps, and $(z)$ is solvable. The other direction uses  \cite[Theorem~C]{MR935009}, which implies that if $z$ is radical then it generates a finite dimensional ideal, so clearly $[z,\pi^{>n}(\varphi)]=0$ for $n$ large enough. The reader should treat \ref{RadDef} as our definition of the radical.


\begin{theorem}[Avramov, Halperin {\cite[Theorem C]{MR896094}}]\label{AHthm} Assuming that $\pd_RS$ is finite, 
$\varphi\colon R\to S$ is a complete intersection homomorphism if and only if every element of $\pi^2(\varphi)$ is  radical.\qed
\end{theorem}

This 
is ultimately how we will establish the complete intersection property. In a future paper we will give a new proof of Theorem \ref{AHthm}, explaining its connection to the cohomological support varieties of Avramov and Buchweitz \cite{MR1794064} and Jorgensen~\cite{MR1972252}. 


\begin{construction}\label{thetaCon}
We fix a minimal model $A$ for $\varphi$ as in \ref{minmod}. An element $z$ in $ \pi^i(\varphi)$ can be thought of as an $R$-linear map $RX_{i-1}\to k$. Lift this to an $R$-linear map $\widetilde{z}\colon RX_{i-1}\to R$. Then, in rough terms, one can differentiate with respect to $\widetilde{z}$ to obtain a degree $1-i$  
 derivation $\dd{z}\colon A\to A$. More precisely, this is the unique $R$-linear derivation which agrees with $\widetilde{z}$ when restricted to the variables; that is, $\dd{z}(x)= \widetilde{z}(x)$ if $x\in X_{i-1}$ while $\dd{z}(x)=0$ if $x\in X_{j}$ for some $j\neq i-1$, and on general elements $\dd{z}$ is determined by the Leibniz rule. This need not be a chain derivation. The boundary 
 \[
 \theta_z\coloneqq [\partial, {\textstyle \dd{z}}]
 \]
 is a degree $-i$ 
 derivation $A\to \m_A$; the fact that $\theta_z$ lands in $\m_A$ will be important later.
 
 Associated to the exact sequence $0\to \m_A\to A \to k\to 0$ there is a connecting homomorphism $\Der_R(A,k)\to \susp \Der_R(A,\m_A)$. Making the identification $\Der_R(A,k)\cong \susp \pi^*(A)$, the assignment $z\mapsto \theta_z$ coincides with this connecting homomorphism, and one can use this as an alternative definition for $\theta_z$. In particular, it follows that up to homotopy $\theta_z$ does not depend on the choice of lift $\widetilde{z}$.
\end{construction}

The next proposition appears implicitly in the proof of \cite[Proposition 4.2]{MR896094}.

\begin{proposition}\label{adProp}
Let $\theta$ be an $R$-linear derivation $A\to \m_A$ of degree $i\leq 0$. By passing to indecomposables, shifting and dualising, the derivation $\theta^{(1)}\colon A^{(1)}\to~A^{(1)}$ induces a degree $i$ map $\pi^*(\varphi)\to \pi^*(\varphi)$. When $\theta=\theta_z$, as in Construction \ref{thetaCon}, this map coincides with $-{\rm ad}(z)$.
\end{proposition}    

\begin{proof}
From the fact that 
$\theta(A)\subseteq \m_{A}$ it follows that $\theta(\m^n_{A})\subseteq \m_{A}^{n}$, and so as well  $\theta^{(1)}(\m^n_{A^{(1)}})\subseteq \m_{A^{(1)}}^{n}$. Hence $\theta^{(1)}$ does induce a map on the indecomposables $\m_{A^{(1)}}/\m_{A^{(1)}}^2$. For the second assertion, a computation using (\ref{hoLieDef}.\ref{bracketdef}) shows 
\[
\langle z\otimes v ,xy\rangle= (-1)^{(i+1)(j+1)}\langle v, {\textstyle \dd{z}}(xy)\rangle
\]
for any $v$ in $ \pi^j(\varphi)$ and $x, y$ in $\m_{A^{(1)}}$. In particular $\langle [z,v], x\rangle $ is
\begin{align*}
    (-1)^j\langle z\otimes v ,\partial(x)\rangle & = (-1)^{1+i(j+1)} \langle v , {\textstyle\dd{z}}(\overline{\partial} x)\rangle \\
    & = (-1)^{1+i(j+1)} \langle v , (\ind \theta_z^{(1)})(x)\rangle\\
    & = -\langle (\ind \theta_z^{(1)})^\vee (v),x\rangle,
\end{align*}
so $ (\ind \theta_z^{(1)})^\vee (v)= -[z,v]$, as was to be shown.
\end{proof}


\begin{lemma}\label{thetaradlem}
Let $\theta,\theta'\colon A\to \m_A$ be  $R$-linear chain derivations of non-positive degree which are cohomologous  inside $\Der_R(A,\m_A)$. Then 
the maps induced on $\pi^*(\varphi)$ by $\theta^{(1)}$ and ${\theta'}^{(1)}$ are the same.

In particular, if there is a derivation $\theta$, cohomologous with $\theta_z$ inside $\Der_R(A,\m_A)$, and an integer $n$ such that the induced derivation $\theta^{(n)}\colon A^{(n)}\to A^{(n)}$ is zero, then $z$ is in the radical of $\pi^*(\varphi)$.
\end{lemma}

\begin{proof}
If $\theta-\theta'=[\partial, \sigma]$  for some derivation $\sigma\colon A\to\m_A$, then $(\theta-\theta')(\m_A)\subseteq \m_A^2$. Hence $\theta$ and $\theta'$ induce the same map on $\ind A^{(1)}$, and then as well on $\pi^*(\varphi)=(\susp \ind A^{(1)})^\vee$.

For the second statement, it follows that $\theta^{(1)}$ induces the map $-\ad(z)$ on $\pi^*(\varphi)$ by Proposition \ref{adProp}. So as well $\theta^{(n)}$ induces the restriction of $-\ad(z)$ to $\pi^{>n}(\varphi)$. If $\theta^{(n)}=0$ then $\ad(z)$ vanishes  on $\pi^{>n}(\varphi)$, so $z$ lies in the radical by \ref{RadDef}.
\end{proof}

Together, the content of Construction \ref{thetaCon}, Proposition \ref{adProp} and Lemma \ref{thetaradlem} can be summarised as follows: there is an assignment  $\Der_R(A, \m_A) \to {\rm End}_k(\pi^* (\phi))$ which induces a well-defined map in cohomology, and moreover the composition
\[
\pi^{*}(\phi) \cong \H^{*-1}(\Der_R(A, k)) \xrightarrow{\ \delta\ } \H^{*}(\Der_R(A, \m_A))\longrightarrow {\rm End}_k(\pi^* (\phi))
\]
sends $z$ to $-\ad(z)$, where $\delta$ is the connecting homomorphism from Construction \ref{thetaCon}.


For the rest of this paper we focus on $\pi^2(\varphi)$. So take $z$ in $\pi^2(\varphi)$ and consider the derivation $\theta_z$ from construction \ref{thetaCon}. 
We consider as well the derivation $\overline{\theta_z}\colon A\to \m_S$ obtained by composing $\theta_z$ with the surjective quasi-isomorphism $\m_A\to \m_S$.


We come to our core technical lemma.

\begin{lemma}\label{corelem}
Suppose there is an $S$-module $M$, finitely generated and of finite projective dimension, and a factorisation  $A \to M\to \m_S$ of $\overline{\theta_z}$, for some $R$-linear chain derivation $A \to M$ and $S$-module map $M\to \m_S$. Then $z$ is radical in $\pi^*(\varphi)$.
\end{lemma}



\begin{proof} 
Let $P$ be a semi-free dg $A$-module resolution of $M$ (\cite{MR2641236} contains all the facts we use about semi-free dg modules). Since $M$ has finite projective dimension over $S$, we can assume that the underlying $A$-module of $P$ is finitely generated and free (i.e.\  it is perfect, see \cite{MR2592508,MR2729016}).

We identify derivations with the corresponding maps out of the dg module of K\"{a}hler differentials. In these terms, our hypotheses give us the commuting outer square of the following diagram of dg $A$-modules
\[
\begin{tikzcd}
    \Omega_{A/R} \ar[rr,"\theta_z"] \ar[dd,two heads,"\simeq"'] \ar[rd,dashed,"\beta"]& & \m_A\ar[dd,two heads,"\simeq"]\\
    & P\ar[d,two heads,"\simeq"] \ar[ur,dashed,"\alpha"] & \\
    S\otimes_A\Omega_{A/R} \ar[r] & M\ar[r] & \m_S.
\end{tikzcd}
\]
Since $\Omega_{A/R}$ and $P$ are semi-free, we can construct lifts $\alpha$ and $\beta$ as shown, making the two trapezia commute. Since the vertical maps are quasi-isomorphisms, the triangle commutes up to homotopy, meaning that $\theta_z$ is cohomologous with $\theta\coloneqq\alpha\beta$.

Since $P$ is  finitely generated and free over $A$
\[
\Hom_A(P,\m_A)\cong \m_A \otimes_A \Hom_A(P,A).
\]
Hence we may write $\alpha$ as a \emph{finite} sum $\alpha=\sum x_i \sigma_i$, for some $x_i$ in $\m_A$ and some $\sigma_i\colon P\to A$ (the $\sigma_i$ need not be chain maps).  We obtain a decomposition $\theta=\sum x_i \sigma_i \beta$. Let $n$ be larger than all of the degrees of the $x_i$ appearing in this sum. Then the induced dg module map $\Omega_{A^{(n)}/k}\to A^{(n)}$ is zero. Therefore the corresponding derivation $
\theta^{(n)}\colon A^{(n)}\to A^{(n)}$ is zero, and  $z$ is in the radical of $\pi^*(\varphi)$ by Lemma~\ref{thetaradlem}.
\end{proof}

Iyengar proved in \cite{MR1707520} that free $S$-module summands of $I/I^2$ give rise to central elements of $\pi^2(\varphi)$. Our next result is a direct analogue for radical elements. There is a natural correspondence
\[
\pi^2(\varphi)\cong \Hom_R(RX_1,k)\cong \Hom_R(I,k)\cong \Hom_S(I/I^2,k)
\]
which we make use of in the statement.

\begin{theorem}\label{radelt}
Let $z$ be in $\pi^2(\varphi)$. If the corresponding homomorphism $I/I^2\to k$ factors through a finitely generated $S$-module $N$ of finite projective dimension, then $z$ is in the radical of $\pi^*(\varphi)$.
\end{theorem}

\begin{proof}
Take an exact sequence $0\to M\to F\to N\to0$ with $F$ a finitely generated free $S$-module and $M\subseteq \m_S F$. By definition $M$ is the first syzygy of $N$, so $\pd_S M$ is finite as well. Recall that by \cite[(2.5)]{MR1269426} there is a minimal presentation
\[
(S\otimes_A\Omega_{A/R})_2\xrightarrow{\partial} (S\otimes_A\Omega_{A/R})_1\to I/I^2\to 0.
\]
We can construct the following commuting diagram
\[
\begin{tikzcd}
    (S\otimes_A\Omega_{A/R})_2 \ar[d,"\partial"] \ar[r,"\beta",dashed] & M \ar[d]\ar[r,"\alpha",dashed] & \m_S \ar[d]\\
    (S\otimes_A\Omega_{A/R})_1 \ar[r,"\delta",dashed]\ar[d] & F \ar[r,"\gamma",dashed]\ar[d]\ar[d] & S \ar[d]\\
    I/I^2 \ar[r] & N \ar[r] & k.
\end{tikzcd}
\]
The lower row is the given factorisation, and the other horizontal maps exist by standard lifting properties. We may take $\widetilde{z}$ in construction \ref{thetaCon} so that $S\otimes_R\widetilde{z}=\gamma\delta$, which implies $\overline{\theta_z}=\alpha\beta$ factors through $M$. By Lemma \ref{corelem} this means $z$ is in the radical of $\pi^*(\varphi)$.
\end{proof}

In the case that $N$ is projective, $z$ corresponds to a free summand of $I/I^2$. The proof shows that in this case $\ad(z)=0$ (since $M=0$), so that $z$ is central, and we recover the result of Iyengar \cite{MR1707520}.

\section{Test modules}
\label{testsection}





\subsection{The conormal module} Theorem \ref{VC} from the introduction can be proven locally, so it is a consequence of the next result. This settles the conjecture of Vasconcelos on the conormal module.


\begin{theorem}\label{sharpVC}
Let $\varphi\colon R\to S=R/I$ be a surjective local homomorphism, and let $k$ be the common residue field.
If there is an $S$-module $N$ of finite projective dimension, and a homomorphism $\alpha\colon I/I^2\to N$ such that $\alpha \otimes_S k$ is injective, then $\varphi$ is a complete intersection.
\end{theorem}

\begin{proof}
The hypotheses imply every homomorphism $I/I^2\to k$ factors through $\alpha$, so Theorem \ref{radelt} every element of $\pi^2(\varphi)$ is radical, and Theorem \ref{AHthm} implies that $\varphi$ is a complete intersection.
\end{proof}

\subsection{The first Koszul homology module} The next result sharpens Theorem~\ref{Kosthm}. It substantially generalises the usual criterion that an ideal is generated by a regular sequence if and only if its first Koszul homology vanishes. 

\begin{theorem}\label{generalKoszul} 
Let $I$ be an ideal of finite projective dimension in a local ring $R$, and let $H=\H_1(I;R)$ be the first Koszul homology of $I$. If every homomorphism $H\to~\!\!\m_{R/I}$ factors through a finitely generated $R/I$-module of finite projective dimension then $I$ is generated by a regular sequence.
\end{theorem}


\begin{proof}
Take a minimal model $A$ for the quotient $\varphi\colon R\to S$. By \cite[(2.5)]{MR1269426} there is a presentation
\[
(S\otimes_A\Omega_{A/R})_3\xrightarrow{\partial} (S\otimes_A\Omega_{A/R})_2\to H\to 0.
\]
Therefore, for any $z$ in $\pi^2(\varphi)$, the chain map $\overline{\theta_z}$ factors as  $S\otimes_A\Omega_{A/R}\to H\to~\m_S$. By hypothesis, it factors further through a finitely generated module of finite projective dimension. By Lemma \ref{corelem} this makes $z$ radical. Since $z$ was arbitrary, Theorem \ref{AHthm} implies that $\varphi$ is a complete intersection homomorphism.
\end{proof}

The case when $\H_1(I;R)$ is assumed to be free was proven by Gulliksen  \cite[Proposition 1.4.9]{MR0262227}. In fact Gulliksen  proves that $I$ is a complete intersection ideal whenever $\H_1(I;R)$ has a non-zero free summand.

Let us write $\tors(H)$ for the torsion submodule of $H$, consisting of those elements annihilated by a non-zero-divisor of $S$, and also  $H^*\coloneqq\Hom_S(H,S)$. 
Every homomorphism  $H\to \m_S$ factors canonically as $H\to H/\tors(H)\to H^{**}\to \m_S$. Therefore Theorem \ref{generalKoszul} implies that $\varphi$ is a complete intersection if either the torsion free quotient $H/\tors(H)$ or the reflexive hull $H^{**}$ have finite projective dimension.

\subsection{Andr\'e-Quillen homology modules} Here we give a generalisation of Theorem \ref{VC} which applies to any homomorphism essentially of finite type. In the statement, $D_i(S/K,-)$ is the $i$th cotangent homology functor defined by Andr\'e and Quillen. 

\begin{theorem}\label{eftversion}
Let $\psi\colon K\to S$ be a homomorphism of commutative noetherian rings, essentially of finite type and of finite flat dimension. If both  $S$-modules $D_1(S/K,S)$ and $D_0(S/K,S)$ have finite projective dimension, then  $\psi$ is locally a complete intersection.
\end{theorem}

In the case that $\psi$ is surjective with kernel $I$ we have $D_0(S/K,S)=\Omega_{S/K}=0$ and $D_1(S/K,S)= I/I^2$, so the statement extends Theorem \ref{VC}.



\begin{proof}
Let $K\to R\to S$ be a factorisation of $\psi$ such that $K\to R$ is smooth, and $R\to S$ is surjective with kernel $I$. By \cite[(3.2)]{MR896094} $\pd_RS$ is finite. The Jacobi-Zariski exact sequence \cite[Theorem 5.1]{MR1726700} takes the form
\[
0\to D_1(S/K,S)\to D_1(S/R,S)\to D_0(R/K,S)\to D_0(S/K,S)\to 0.
\]
The $S$-module $D_0(R/K,S)$ is free, since $R$ is smooth over $K$,  and therefore if  $D_1(S/K,S)$ and $D_0(S/K,S)$ have finite projective dimension so does $D_1(S/R,S)\cong I/I^2$. By Theorem \ref{VC}, $R\to S$ is locally a complete intersection, and this means by definition that $\psi$ is locally a complete intersection homomorphism. 
\end{proof}


\subsection{The module of K\"ahler differentials}\label{Kahlersection} In this subsection, $S$ is a local algebra with residue field $k$, essentially of finite type and having finite flat dimension over a commutative noetherian ring $K$.

In the case that $K$ is a field, Vasconcelos conjectured that under some appropriate separability conditions 
the finiteness of $\pd_S \Omega_{S/K}$ should imply $S$ is a complete intersection \cite{MR508082}. In the present context, this conjecture becomes the question: \emph{under what circumstances can one remove the assumption on $D_1(S/K,S)$ in Theorem \ref{eftversion}?} It turns out this question is closely related to a conjecture of Eisenbud and Mazur \cite{MR1465370}.

An \emph{evolution} of $S$ is a local $K$-algebra $T$, essentially of finite type, and a local $K$-algebra surjection $T\to S$ such that $S\otimes_T\Omega_{T/K}\to \Omega_{S/K}$ is an isomorphism.

These maps arose in the work of Scheja and Storch  \cite{MR263808} and B\"oger \cite{MR282966}, and have since proven  important in commutative algebra, especially in relation to symbolic powers (see \cite{MR3779569}). They appeared in the the work of Wiles \cite{MR1333035}, where it was crucial that a certain algebra did not have any evolutions, other than isomorphisms. That the following holds is now known as the Eisenbud-Mazur conjecture: \emph{if $K$ is either a field of characteristic zero, or a mixed characteristic discrete valuation ring, and $S$ is a flat, reduced $K$-algebra, then $S$ does not admit nontrivial evolutions}.






\begin{theorem}\label{cotthm}
Assume that $S$ has no non-trivial evolutions. If the module of K\"ahler differentials $\Omega_{S/K}$ has finite projective dimension, then $S$ is a complete intersection over $K$.
\end{theorem}

It is well-known that (as long as $S$ is flat and generically separable over $K$) if $\pd_S\Omega_{S/K}$ is finite then $S$ is reduced. Therefore, the theorem  asserts that Vasconcelos' conjecture on the K\"ahler differentials is a consequence of the Eisenbud-Mazur conjecture. Since Platte has shown that $S$ must be quasi-Gorenstein if $\pd_S\Omega_{S/K}$ is finite \cite{MR595423}, it would suffice to prove the Eisenbud-Mazur conjecture for quasi-Gorenstein rings (that is, rings having a free canonical module). The Eisenbud-Mazur conjecture is known in various cases (cf.~\cite{MR1465370}), so we obtain various new cases for Vasconcelos' conjecture.

We deduce Theorem \ref{cotthm} using a  result of Lenstra.

\begin{proposition}[Lenstra \cite{MR1465370}]\label{Lenstra}
Choose a presentation $S=R/I$, where $R$ is a localisation of a polynomial $K$-algebra. The following are equivalent
\begin{enumerate}
    \item every evolution of $S$ is an isomorphism;
    \item\label{minseq} in the exact sequence 
    \[
    I/I^2 \xrightarrow{\ d\ }  S\otimes_R\Omega_{R/K}\to \Omega_{S/K}\to 0
    \]
    no minimal generator of $I/I^2$ is mapped to zero by $d$.\qed
\end{enumerate}
\end{proposition}

\begin{proof}[Proof of Theorem \ref{cotthm}]
%
Let $\varphi\colon R\to S$ as in Proposition \ref{Lenstra}. By \cite[(3.2)]{MR896094} $\pd_RS$ is finite. 
By definition, $S$ is a complete intersection over $K$ if and only if $\varphi$ is a complete intersection homomorphism.

We take the first syzygy of $\Omega_{S/K}$, defined by the exact sequence
\[
0\to N\to S\otimes_R\Omega_{R/K}\to \Omega_{S/K}\to 0.
\]
Then $N$ has finite projective dimension because $\Omega_{S/K}$ does. 
The exact sequence in Proposition \ref{Lenstra} part (\ref{minseq})
induces a map $\alpha\colon I/I^2\to N$, and (\ref{minseq}) says exactly that $\alpha\otimes_S k$ is injective. Therefore we are in the situation of Theorem \ref{sharpVC}, and $\varphi$ is a complete intersection homomorphism. 
\end{proof}

\bibliographystyle{amsplain}
\bibliography{Vasconcelos}

\end{document}